\documentclass[11pt]{amsart}
\usepackage{amssymb,latexsym, amsmath, amsxtra}
\usepackage{graphicx}

\newtheorem{theorem}{Theorem}
\newtheorem{lemma}{Lemma}

\begin{document}

\author{Estelle Basor, Brian Conrey, Kent E. Morrison}
\address{American Institute of Mathematics, San Jose, CA}
\title{Knots and ones}
\date{March 2, 2017}
\subjclass[2010]{Primary 57M27; Secondary 05A10,11B65}

\begin{abstract}We give a number theoretic proof of the integrality of certain BPS invariants of knots. The formulas for these numbers are sums involving binomial coefficients and the M\"obius function. We also prove a conjecture about further divisibility properties of the  invariants.
\end{abstract}

\maketitle
\large
\renewcommand{\baselinestretch}{1.1}   
\normalsize

A recent SQuaREs group at AIM asked us if we could prove directly that $n_r$ is always an integer where
$$n_r= r^{-2} \sum_{d\mid r} \mu\left(\frac r d\right) \left(3 d-1 \atop d-1\right).$$
It was already known that the $n_r$ are integers by work of Garoufalidis, Kucharski, 
and Su\l kowski \cite[Proposition 1.2]{GKS2016} in which the sums are identified as the formulas for the extremal BPS invariants of certain twist knots, but the question put to us was whether the integrality could be proved from the formula. In addition, we noticed that often $n_r/r$ was also an integer and that $2n_r/r$ appeared always to be an integer. The observation that $2n_r/r$ is an integer is one case of the same authors' ``Improved Integrality'' conjecture \cite[Conjecture 1.3]{GKS2016}. In this article we prove that and describe exactly when $n_r/r$ is not an integer.

 We show that if $r$ is odd then $n_r/r$ is integral. Empirically  somewhere around $32\%$ of even $r$ 
have $n_r/r$ integral.
For even $r$, the results are a little complicated to state;   
the case that $r=2p$ where $p$ is a prime is fairly easy to describe.
The initial list of primes for which $n_{2p}/(2p)$ is an integer are 
\begin{multline*}\{5, 17, 37, 41, 73, 137, 149, 257, 277, 293, 337, 521, 577, 593,\\ 641, 
661, 673, 677, 1033, 1061,1093,1097 \dots\}
\end{multline*}
 We can prove that  
 $n_{2p}/(2p)$ is integral precisely for  the primes $p$  which have no consecutive ones in their binary 
expansion. 
For example $ 10001001001_2=1097$ is on the list but 
$100011001_2=281$ is not. We say that $r$ is binarily ``well-spaced'' if its binary expansion does not have 
consecutive 1's. More precisely a number 
$$r=\sum_{j=0}^J \epsilon_j 2^j$$
with $\epsilon_j\in \{0,1\}$ is well-spaced if and only if $\epsilon_j \epsilon_{j+1}=0$ for $j=0,1,\dots ,J-1$.

Turning to more general even $r$ we let
$\mathcal W$ be the set of odd integers which are binarily well-spaced
\begin{eqnarray*} &&
\mathcal W=\{1, 5, 9, 17, 21, 33, 37, 41, 65, 69, 73, 81, 85, 129, 133, 137, 145, \\
&&
149, 161, 165, 169, 257, 261, 265, 273, 277, 289, 293, 297, 321, 325, \\
&&
329, 337, 341, 513, 517, 521, 529, 533, 545, 549, 553, 577, 581, 585, \\
&&
593, 597, 641, 645, 649, 657, 661, 673, 677, 681,\dots\}
\end{eqnarray*}
Our basic result is the following.
\begin{theorem}
For any $r$ the ratio 
$\frac{2n_{r}}{ r}$
is an integer. Moreover, 
$\frac{n_{r}}{r}$
is an integer if and only if either $r$ is odd or else 
 $r=2^kr_1$ where $k>0$ and $r_1$ is odd and the set of $w\in \mathcal W$ for which $w\mid r_1$ and $\frac {r_1} w $ is squarefree has even cardinality. 
\end{theorem}

\section{Opening arguments}
Suppose that the exact power of $p$ which divides $r$ is $p^k$, i.e. $p^k\mid r$ and 
$p ^{k+1}\nmid r$. Then we can write $r=p^kr_1$ where $p\nmid r_1$ and
\begin{eqnarray} \label{eqn:open}
r^2 n_r = \sum_{d\mid r_1}  
 \mu\big( \frac {r_1}{d}\big)\left(\left( 3p^{k}d-1 \atop p^{k}d-1\right)-\left( 3p^{k-1}
d-1 \atop p^{k-1}d-1
\right)
\right).
\end{eqnarray}
The expression in $(\dots)$ may be written as 
\begin{eqnarray*}
\left( 3p^{k-1}
d-1 \atop p^{k-1}d-1
\right) \times \left(\frac{\prod_{j\le dp^k\atop p\nmid j}(x+j)}{\prod_{j\le dp^k \atop p\nmid j}j}-1\right)
\end{eqnarray*}
where $x=2p^kd$.  It suffices to show that
\begin{eqnarray} \label{eqn:cong}
\frac{\prod_{j\le dp^k\atop p\nmid j}(x+j)}{\prod_{j\le dp^k \atop p\nmid j}j}\equiv 1 \bmod p^{3k}.
\end{eqnarray}
The coefficient of $x$ on the left hand side is 
\begin{eqnarray} \label{eqn:x}
\sum_{j\le dp^k\atop p\nmid j} \frac 1j
\end{eqnarray}
and the coefficient of $x^2$ is 
\begin{eqnarray} \label{eqn:x2}
\sum_{1\le i<j\le dp^k\atop p\nmid ij}
\frac{1}{ij}.
\end{eqnarray}
At this point we consider odd $p$ versus $p=2$.

\section{Integrality of $n_r/r$ for odd $r$}

Assume in this section that $p$ is odd.  We have 
\begin{eqnarray*}\sum_{j\le dp^k\atop p\nmid j} \frac 1j
=\sum_{j < dp^k/2\atop p\nmid j}\left(\frac{1}{j}+\frac{1}{dp^k-j}\right)
=dp^k \sum_{j < dp^k/2\atop p\nmid j} \frac{1}{j(dp^k-j)} .
\end{eqnarray*}
Since $p^k\mid x$ it remains only to show that 
\begin{eqnarray*}
\sum_{j < dp^k/2\atop p\nmid j} \frac{1}{j(dp^k-j)} 
+2\sum_{1\le i<j\le dp^k\atop p\nmid ij}
\frac{1}{ij}
\equiv 0 \bmod p^k,
\end{eqnarray*}
which is equivalent to 
\begin{eqnarray*}
\sum_{j < dp^k/2\atop p\nmid j} \frac{1}{j^2} 
\equiv 2\sum_{1\le i<j\le dp^k\atop p\nmid ij}
\frac{1}{ij}
  \bmod p^k.
\end{eqnarray*}
But
\begin{eqnarray*}
\sum_{j < dp^k/2\atop p\nmid j} \frac{1}{j^2 } \equiv \frac 12  
\sum_{j < dp^k\atop p\nmid j} \frac{1}{j^2 } \equiv \frac 1 2 
\sum_{j < dp^k\atop p\nmid j} j^2\bmod p^k.
\end{eqnarray*}
 And
 \begin{eqnarray*}
2\sum_{1\le i<j\le dp^k\atop p\nmid ij}
\frac{1}{ij}\equiv 2 \sum_{1\le i<j\le dp^k\atop p\nmid ij}
ij  \equiv  \bigg(\sum_{1\le i \le dp^k\atop p\nmid i}
i\bigg)^2- \sum_{1\le i \le dp^k\atop p\nmid i}
i^2 \bmod p^k.
\end{eqnarray*}
 
 But 
 \begin{eqnarray*}
 \sum_{1\le i \le dp^k\atop p\nmid i}
i \equiv 0 \bmod p^k
\end{eqnarray*}
so it remains to show that 
\begin{eqnarray*}
\frac 1 2 
\sum_{j < dp^k\atop p\nmid j} j^2\equiv 
- \sum_{1\le i \le dp^k\atop p\nmid i}
i^2 \bmod p^k.
\end{eqnarray*}
This is obvious if $p>3$  just by summing. If $p=3$ then we need to use the fact that
$$\frac 1 2 \equiv -1 \bmod 3.$$

It follows that $n_r/r$ is an integer for any odd $r$.

\section{Twos}
We now turn to the divisibility of $n_r$ by powers of 2. We begin by showing that $2n_r/r$ is an integer.
Let $r=2^kr_1$ where $r_1$ is odd. We have
\begin{eqnarray} \label{eqn:open2}
r^2 n_r &=& \sum_{d\mid r_1}  
 \mu\big( \frac {r_1}{d}\big)\left(\left( 3\cdot 2^{k}d-1 \atop 2^{k}d-1\right)-\left( 3\cdot 2^{k-1}
d-1 \atop 2^{k-1}d-1
\right)
\right).
\end{eqnarray}
The expression in $(\dots)$ may be written as 
\begin{eqnarray*}
\left( 3\cdot 2^{k-1}
d-1 \atop 2^{k-1}d-1
\right) \times \left(\frac{\prod_{j\le d\cdot 2^k\atop 2\nmid j}(x+j)}{\prod_{j\le dp^k \atop 2\nmid j}j}-1\right)
\end{eqnarray*}
where $x=2^{k+1}d$. 
The quantity in $(\dots )$ is 
\begin{eqnarray} \label{eqn:x}
x\sum_{j\le d2^k\atop 2\nmid j} \frac 1j
+x^2 \sum_{1\le i<j\le d2^k\atop 2\nmid ij}
\frac{1}{ij} +O(x^3)
\end{eqnarray}
Now
\begin{eqnarray*}
\sum_{j\le d2^k\atop 2\nmid j} \frac 1j=\sum_{j\le d2^{k-1}\atop 2\nmid j} \bigg( \frac 1j +\frac{1}{2^kd-j}\bigg)
=2^kd\sum_{j\le d2^{k-1}\atop 2\nmid j} \frac 1{j(2^kd-j)}
\end{eqnarray*}
and
\begin{eqnarray*}
2\sum_{j\le d2^{k-1}\atop 2\nmid j} \frac 1{j(2^kd-j)}\equiv -
\sum_{j\le d2^{k}\atop 2\nmid j} \frac 1{j^2}\bmod 2^k
\end{eqnarray*}
Therefore (\ref{eqn:x}) is $2^{2k}$ times a quantity which is 
\begin{eqnarray*} 
\equiv \bigg(-d^2\sum_{j\le d2^k\atop 2\nmid j} \frac 1{j^2}
+4d^2\sum_{1\le i<j\le d2^k\atop 2\nmid ij}
\frac{1}{ij} \bigg)\bmod 2^{k}
\end{eqnarray*}
 Replacing $i$ and $j$ by $1/i$ and $1/j$ 
the above is
\begin{eqnarray*} \equiv \bigg(-d^2\sum_{j\le d2^k\atop 2\nmid j}j^2
+4d^2\sum_{1\le i<j\le d2^k\atop 2\nmid ij}
ij\bigg)\bmod 2^{k}
\end{eqnarray*}
Since $d$ is odd, it is easily checked that this is
$$\equiv 2^{k-1} \bmod 2^k.$$
  Thus we have established that $2n_r/r$ is always integral.

\section{When is the factor of 2 actually needed?}
To begin to analyze this question, we introduce $b(t)$ which is the sum of the binary digits of $t$. In other 
words $b(t)$ is the number of ones in the expression of $t$ in base 2. 
Then we can prove:
{\it If $p$ is prime then $n_{2p}/(2p)$ is an integer if and only if 
$$2 b(p) =b(3p-1)+1.$$
Moreover the only primes $p$ which satisfy this condition are those that are well-spaced. }

\subsection{Sketch of proof.}
To study the integrality of $n_{2p}/(2p)$ we have to  consider
\begin{eqnarray*}
1-5-\left(3p-1\atop p-1\right) +\left( 6p-1\atop 2p-1\right)
\end{eqnarray*}
modulo $8$ and modulo $p^3$.  Let's consider it modulo 8. What is the exponent  of 2 in the prime factorization of each of these? 
Recall that the power of 2 that divides $t!$ is $t-b(t)$. Therefore the power of 2 which divides 
$$ \left(3p-1\atop p-1\right)$$
is 
$$b(p-1)+b(2p)-b(3p-1)=b(p)-1+b(p) -b(3p-1)=2b(p)-1-b(3p-1) $$
since for any odd prime it is easy to show that $b(p)=b(p-1)+1$.
Similarly, the exponent of 2 in 
$$ \left( 6p-1\atop 2p-1\right)  $$
is
\begin{eqnarray*}
b(2p-1)+b(4p)-b(6p-1)&=&b(2(p-1)+1)+b(p)-b(2(3p-1)+1)\\&=&b(p-1)+1 +b(p)-b(3p-1)-1\\&=& 2b(p)-1-b(3p-1)
\end{eqnarray*}
which is the same power of 2 as for $({3p-1\atop p-1})$.

 Now it follows from (\ref{eqn:b3d1}) and  (\ref{eqn:b3d2}) below  that for an odd prime $p$ it is never the case that $2b(p)-1-b(3p-1)=1$. 
 Using this, then the only case in
question is when $2b(p)-1-b(3p-1)=0$. In this case it always holds that 
$-({{3p-1}\atop {p-1}}) + ({{6p-1}\atop {2p-1}})\equiv 4 \bmod 8$.
This leads to the proof that 
$n_{2p}/{2p}$ is integral precisely when $2b(p)=b(3p-1)+1$.

\section{A basic lemma}
In this section we prove that for any $d$ the exact power of 2 dividing $({3d-1\atop d-1})$ is the same as that dividing
$({6d-1\atop 2d-1})$; moreover this exact power is never 1 and it is 0 precisely when the binary digits of $d$ are well spaced.

\begin{lemma} \label{lem:basic}
$$
\left(3d-1 \atop d-1\right) \ne \left( 6d-1 \atop 2d-1\right) \bmod 8$$
if and only if $d$ is odd and the binary digits of $d$ are well-spaced.
\end{lemma}
\begin{proof}
The proof has several parts. First we prove that the exact power of 2 dividing each of the above binomial coefficients 
is the same. Then we prove that that exact power is never 1. Then we prove that in the case that 
both are odd, then they are inequivalent mod 8. Finally, we show that it is precisely the odd $d$ with well spaced binary digits which
give odd values of the two binomial coefficients. 

Recall that the power of 2 that divides $t!$ is $t-b(t)$. Therefore the power of 2 which divides 
$$ \left(3d-1\atop d-1\right)$$
is 
$$b(d-1)+b(2d)-b(3d-1) =b(d-1)+b(d)-b(3d-1) $$
since for any $d$ it is clear that $b(2d)=b(d)$.
Similarly, the exponent of 2 in 
$$ \left( 6d-1\atop 2d-1\right)  $$
is
\begin{eqnarray*}
b(2d-1)+b(4d)-b(6d-1)&=&b(2(d-1)+1)+b(d)-b(2(3d-1)+1)\\&=&b(d-1)+1 +b(d)-b(3d-1)-1\\&=& b(d-1)+b(d)-b(3d-1)
\end{eqnarray*}
which is the same power of 2 as for $({3d-1\atop d-1})$.

Now consider the situation where $d$ is odd and has a well-spaced binary
expansion. Then $d-1$ is even and so 
$$b(d-1)+1 =b(d).$$
Also, $d-1$ has well spaced binary digits, too, since its expansion is the same as that 
of $d$ but with the final 1 replaced by a 0. Now, multiplication of a well spaced number by $3=11_2$
results is a number with twice the binary digit sum, because there is no carrying. In other words, for 
any well-spaced $d$ it is the case that 
\begin{eqnarray*} \label{eqn:b3d} b(3d)=2 b(d).
\end{eqnarray*}
Now we are assuming that $d$ is odd; therefore $3d$ is odd so that 
\begin{eqnarray} \label{eqn:b3d1} b(3d-1)=b(3d)-1=2b(d)-1.
\end{eqnarray}
Therefore, 
$$b(3 d-1)=b(d-1)+b(d)$$
as claimed. 

Now suppose that $d$ is odd but its binary expansion is not well-spaced.
If $d$ is not well-spaced, then 
\begin{eqnarray} \label{eqn:b3d2} b(3d)<2b(d)-1.\end{eqnarray}
 To see this imagine the base 2 multiplication:
$$(1\dots 011\dots 1 0 \dots 1)_2\times (11)_2$$
We have isolated a block of 1's with a zero at the left.  
Therefore,
$$b(3d-1)=b(3d)-1< 2b(d)-1 = b(d)+b(d-1).$$\end{proof}

 \section{The case where  $r$ is twice an odd number}
In this case $n_r/r$ will be integral if and only if 8 divides
\begin{eqnarray*}
\sum_{d\mid \frac{r}{2}}  
 \mu\big( \frac {r}{2d}\big)\left(\left( 6d-1 \atop2d-1\right)-\left( 3 
d-1 \atop d-1
\right)
\right).
\end{eqnarray*}
Now for each  $d\mid \frac r 2$ for which $\frac{r}{d}$ is squarefree and $d$ is binarily well-spaced, by Lemma \ref{lem:basic} we get a contribution 
of $4 \bmod 8$ to the sum above. Consequently the sum will be 0 modulo 8 precisely when there are an even number of such terms. 

\section{Final arguments}
Now suppose that $r=2^kr_1$ where $r_1$ is odd. Then by (\ref{eqn:open}) $n_r/r$ is integral precisely when 
\begin{eqnarray*}
 \sum_{d\mid r_1}  
 \mu\big( \frac {r_1}{d}\big)\left(\left( 3\cdot 2^{k}d-1 \atop 2^{k}d-1\right)-\left( 3\cdot 2^{k-1}
d-1 \atop 2^{k-1}d-1
\right)
\right)\equiv 0 \bmod 2^{3k}
\end{eqnarray*}
By our work in the section on twos this is 
\begin{eqnarray*}
\equiv 2^{3k-1}\sum_{d\mid r_1}  
 \mu\big( \frac {r_1}{d}\big)\left( 3\cdot 2^{k-1}d-1 \atop 2^{k-1}d-1\right)\bmod 2^{3k}.
\end{eqnarray*}
The terms for which 
\begin{eqnarray*}
\left( 3\cdot 2^{k-1}d-1 \atop 2^{k-1}d-1\right)
\end{eqnarray*}
is even will end up contributing $0\bmod 2^{3k}$, so we just need to account for the terms when this is odd. 
By Lemma \ref{lem:basic} these will be the terms for which $2^{k-1}d$ is binarily well-spaced. Clearly these are the terms for which 
$d$ is binarily well-spaced. Hence our theorem is proven.

\end{document}